\documentclass[11pt]{article}
\usepackage{amsthm, amsmath, amssymb, amsfonts, url, booktabs, tikz, setspace, fancyhdr, bm}
\usepackage[hidelinks]{hyperref}
\usepackage{geometry}
\usepackage{hyperref, enumerate}
\usepackage[shortlabels]{enumitem}
\usepackage[babel]{microtype}
\usepackage[english]{babel}
\usepackage[capitalise]{cleveref}
\usepackage{comment}
\usepackage{bbm}
\usepackage{csquotes}
\usepackage{diagbox}
\usepackage{mathabx}
\usepackage{tikz}

\counterwithin{figure}{section}

\newtheorem{theorem}{Theorem}[section]
\newtheorem{proposition}[theorem]{Proposition}
\newtheorem{lemma}[theorem]{Lemma}

\theoremstyle{definition}
\newtheorem{problem}[theorem]{Problem}

\newtheorem{question}[theorem]{Question}

\newtheorem{remark}[theorem]{Remark}

\newlist{Case}{enumerate}{2}
\setlist[Case, 1]{%
    label           =   {\bfseries Case \arabic*.},
    labelindent=1em ,labelwidth=1.3cm, labelsep*=1em, leftmargin =!
}
\setlist[Case, 2]{%
    label           =   {\bfseries Subcase \arabic{Casei}.\arabic*.},
    labelindent=-1em ,labelwidth=1.3cm, labelsep*=1em, leftmargin =!
}

\usepackage{todonotes}

\title{Local rainbow colorings for various graphs}
\author{
Xinbu Cheng\thanks{Laboratory of Mathematics and Complex Systems, Ministry of Education, School of Mathematical Sciences, Beijing Normal University, Beijing, China. Emails: chengxinbu2006@sina.com.}
\and
Zixiang Xu\thanks{Extremal Combinatorics and Probability Group (ECOPRO), Institute for Basic Science (IBS), Daejeon, South Korea. Email: zixiangxu@ibs.re.kr. Supported by IBS-R029-C4.}
}
\date{}

\begin{document}

\maketitle
\begin{abstract}
Motivated by a problem in theoretical computer science suggested by Wigderson, Alon and Ben-Eliezer studied the following extremal problem systematically one decade ago. Given a graph $H$, let $C(n,H)$ be the minimum number $k$ such that the following holds. There are $n$ colorings of $E(K_{n})$ with $k$ colors, each associated with one of the vertices of $K_{n}$, such that for every copy $T$ of $H$ in $K_{n}$, at least one of the colorings that are associated with $V(T)$ assigns distinct colors to all the edges of $E(T)$. In this paper, we obtain several new results in this problem including:
\begin{itemize}
    \item For paths of short length, we show that $C(n,P_{4})=\Omega(n^{1/5})$ and $C(n,P_{t})=\Omega(n^{1/3})$ with $t\in\{5,6\}$, which significantly improve the previously known lower bounds $(\log{n})^{\Omega(1)}$.
    \item We make progress on the problem of Alon and Ben-Eliezer about complete graphs, more precisely, we show that $C(n,K_{r})=\Omega(n^{2/3})$ when $r\geqslant 8$. This provides the first instance of graph for which the lower bound goes beyond the natural barrier $\Omega(n^{1/2})$. Moreover, we prove that $C(n,K_{s,t})=\Omega(n^{2/3})$ for $t\geqslant s\geqslant 7$.
    \item When $H$ is a star with at least $4$ leaves, a matching of size at least $4$, or a path of length at least $7$, we give the new lower bound for $C(n,H)$. We also show that for any graph $H$ with at least $6$ edges, $C(n,H)$ is polynomial in $n$. All of these improve the corresponding results obtained by Alon and Ben-Eliezer.
\end{itemize}


\smallskip

\end{abstract}

\section{Introduction}
One of the hardest problems of complexity theory is to prove nontrivial lower bounds on fundamental complexity measures for concrete computing problems. In 1993, Karchmer~\cite{1993IEEESP} gave a lower bound on non-deterministic circuit size and presented a new proof for the exponential monotone size lower bound for the clique function. Later, Wigderson~\cite{1993Wigderson} discussed the achievements, potential, and challenges of the elegant fusion method introduced by Karchmer~\cite{1993IEEESP}, which unifies the previous approximation method of Razborov~\cite{1989STOC} and the topological method of Sipser~\cite{1985Sipser}. In the same paper, Wigderson also provided several concrete open problems, one of which can be stated as follows.

\begin{problem}\label{problem:TCS}
Let $\chi_{i}: \{0,1\}^{n}\rightarrow [k]$, for $i\in [n]$ be a collection of $k$-colorings of the $n$-dimensional hypercube. For a triple of distinct vectors $X,Y,Z\in\{0,1\}^{n}$, say that coordinate $i\in [n]$ is interesting if not all three vectors agrees in this coordinate. Say that $\chi_{i}$ is proper on this triple if the three colors $\chi_{i}(X)$, $\chi_{i}(Y)$ and $\chi_{i}(Z)$ are distinct. Define the collection of colorings good if for every triple of vectors, there is an interesting coordinate $i$ for which $\chi_{i}$ is proper on this triple. The problem is that, bound the smallest number $k$ for which such a good collection exists.   
\end{problem}

Karchmer and Wigderson~\cite{1993Karchmer} later proved that, in the above problem, the smallest number $k$ has to grow with $n$. Motivated by Problem~\ref{problem:TCS}, Alon and Ben-Eliezer~\cite{2011AlonJC} initiated the study of a new problem on extremal graph theory, which aims to find rainbow subgraphs under certain constraints. Formally, for a given graph $H$, let $C(n,H)$ be the minimum number $k$ such that the following holds. There is a set of $n$ colorings $\{f_{v}:E(K_{n})\rightarrow [k]:v\in V(K_{n})\}$, such that for every copy $T$ of $H$ in $K_{n}$, at least one of the colorings that are associated with $V(T)$ assigns distinct colors to all the edges of $E(T)$, that is, at least one vertex can induce a rainbow copy of $H$. In particular, we do not require each coloring to be proper edge coloring.

In recent years, there are also many other important problems in the field of extremal combinatorics involving finding rainbow structures in edge colorings of complete graphs. For example, the rainbow Tur\'{a}n problem~\cite{2022Liu,  2013EUJCDas, Janzer2020, 2017EJCJohnston, 2007CPCRainbow, 2022Tomon, 2022Wang} asks the maximum number of edges in a properly edge-colored complete graph that does not contain certain subgraph, all of whose edges have different colors. The anti-Ramsey problem~\cite{2015Ars, 1975ErdosAnti, 2002JCTBJiang, 2020DMYuan, 2021arxivYuan} asks for the maximum number of colors in an edge coloring of a complete graph with no certain rainbow subgraph. Moreover, there was a famous conjecture of Ringel in 1963, one of whose statements involved finding a rainbow copy of any tree with $n$ edges in a particular proper edge coloring of $K_{2n+1}$. This conjecture was recently confirmed by Montgomery, Pokrovskiy, and Sudakov~\cite{2021GAFASudakov} via many new techniques that are based on probabilistic methods. Return to the new extremal problem of Alon and Ben-Eliezer, one can ask the following natural question.
\begin{question}
For a fixed graph $H$, determine the order of growth of $C(n,H)$ as $n\rightarrow\infty$.
\end{question}

Alon and Ben-Eliezer~\cite{2011AlonJC} characterized the set of all graphs $H$ for which $C(n,H)$ is bounded by some absolute constant $c(H)$. Using the so-called Lov\'{a}sz local lemma~\cite{2016AlonProb, 1975Locallemma}, they proved a general upper bound for any graph $H$. Moreover, they also obtained lower bounds for several graphs of special interests, including paths $P_{t}$, matchings $I_{t}$, and stars $S_{t}$. Here we list some known results in~\cite{2011AlonJC} as follows.

\begin{theorem}[\cite{2011AlonJC}]\label{thm:alon}\
\begin{itemize}
    \item $C(n,H)\leqslant c(H)$ if and only if $H$ contains at most $3$ edges and $H$ is neither $P_{3}$ nor $P_{3}$ together with any number of isolated vertices. Moreover, in all these cases, we have $C(n,H)\leqslant 5$ for every $n$.
    \item Let $H$  be a fixed graph with $r$ vertices, then $C(n,H)=O(r^{4}\cdot n^{\frac{r-2}{r}})$.
    \item $C(n,P_{3})=\Omega((\frac{\log{n}}{\log\log{n}})^{1/4})$ and $C(n,P_{t})=(\log{n})^{\Omega(1)}$ for $t\in\{4,5,6\}$.
    \item $C(n,I_{4})=\Omega(n^{1/6})$ and $C(n,I_{t})=\Omega(n^{1/4})$ for $t\geqslant 5$.
    \item $C(n,S_{4})=\Omega(n^{1/4})$ and $C(n,S_{t})=\Omega(n^{1/3})$ for $t\geqslant 5$.
    \item $C(n,P_{7})$, $C(n,P_{8})=\Omega(n^{1/6})$ and $C(n,P_{t})=\Omega(n^{1/4})$ for $t\geqslant 9$.
    \item For any graph $H$ with at least 13 edges, there is a constant $b=b(H)>0$ so that $C(n,H)=\Omega(n^b)$.
\end{itemize}
\end{theorem}

In this paper, we show some new lower bounds on $C(n,H)$ for various graphs, including several sparse graphs such as paths, stars, and matchings, and dense graphs such as cliques and complete bipartite graphs.

For complete graphs, Alon and Ben-Eliezer~\cite{2011AlonJC} asked whether for every $\epsilon>0$, there is an $r=r(\epsilon)$ such that $C(n,K_{r})=\Omega(n^{1-\epsilon})$. However, they did not provide any good bound for $C(n,K_{r})$. As we can see that the natural counting argument hits a barrier at $\Omega(n^{1/2})$, here we break this barrier and make partial progress to their conjecture by showing the following result.

\begin{theorem}\label{thm:cliques}
For any positive integer $r\geqslant 8$, we have
\begin{equation*}
    C(n,K_{r})=\Omega(n^{2/3}).
\end{equation*}
\end{theorem}
Furthermore, we can also prove a new bound for the complete bipartite graphs as follows.

\begin{theorem}
For any positive integers $t\geqslant s\geqslant 7$, we have
\begin{equation*}
    C(n,K_{s,t})=\Omega(n^{2/3}).
\end{equation*}
\end{theorem}

Our improved lower bounds for sparse graphs can be listed as follows.

\begin{theorem}\label{thm:paths}
Let $P_{t}$ be the path of length $t$, we have\
\begin{itemize}
    \item $C(n,P_{4})=\Omega(n^{1/5})$.
    \item $C(n,P_{t})=\Omega(n^{1/3})$, for $t\in\{5,6,7\}$.
    \item $C(n,P_{t})=\Omega(n^{1/2})$, for $t\geqslant 8$.
\end{itemize}
\end{theorem}

Theorem~\ref{thm:paths} resolves the problem for almost all paths, leaving only the case of $P_{3}$ open. 

\begin{theorem}\label{thm:stars}
Let $S_{t}$ be the star with $t$ leaves, we have\
\begin{itemize}
    \item $C(n,S_{4})=\Omega(n^{1/3})$.
    \item $C(n,S_{t})=\Omega(n^{1/2})$, for $t\geqslant 5$.
\end{itemize}
\end{theorem}

\begin{theorem}\label{thm:matchings}
Let $I_{t}$ be the matching of size $t$, we have\
\begin{itemize}
    \item $C(n,I_{4})=\Omega(n^{1/5})$.
    \item $C(n,I_{t})=\Omega(n^{1/3})$, for $t\in\{5,6\}$.
    \item $C(n,I_{t})=\Omega(n^{1/2})$, for $t\geqslant 7$.
\end{itemize}
\end{theorem}

The next result shows that if $H$ has at least $6$ edges, then $C(n,H)$ must be polynomial in $n$. 

\begin{theorem}\label{thm:666}
For any graph $H$ with at least $6$ edges, there is a constant $b=b(H)>0$ so that $C(n,H)=\Omega(n^b)$.
\end{theorem}
This improves the result of Alon and Ben-Eliezer from $13$ edges to $6$ edges. Note that the first result in Theorem~\ref{thm:alon} tells that the constant cannot be improved to $3$, thus our result is very close to being optimal.

\medskip

\noindent\textbf{Notation.} Throughout this paper, we use $f$ to denote the set of coloring functions associated with the corresponding vertices. We will denote $H_{1}\sqcup H_{2}$ as the vertex-disjoint union of the graphs $H_{1}$ and $H_{2}$. In particular, we will write $H\sqcup\{p\}$ for the graph which consists of $H$ plus an isolated vertex $p$. We always assume $n$ is a sufficiently large number. The notations $O(\cdot)$, $\Omega(\cdot)$ and $o(\cdot)$ have their usual asymptotic meaning. We omit the flooring and 
ceiling functions if not essentially. We may abuse some letters or mathematical symbols, and in each section, the meaning of each letter and the mathematical symbol will be clear.

\medskip
\noindent\textbf{Structure of the paper.}
The rest of this paper is organized as follows. We will present some auxiliary lemmas in Section~\ref{sec:pre}. The proofs of new results for cliques and complete bipartite graphs are presented in Section~\ref{sec:dense}. We prove the lower bounds for paths, stars and matchings separately in Section~\ref{sec:Sparse}. We will show the polynomial lower bound for any graph with at least $6$ edges in Section~\ref{sec:C4}. Finally, we conclude this paper and pose some open problems in Section~\ref{sec:conclusions}.

\section{Preliminaries}\label{sec:pre}
To show $C(n,H)> k$, our main task is to show that for any set of $n$ coloring functions with $k$ colors, we can always find a copy of $H$ such that none of its vertices induces a rainbow coloring. Moreover, one can see that if $H'\subseteq H$ is a subgraph of $H$ on the same set of vertices, then every lower bound for $C(n,H')$ implies the same lower bound for $C(n,H)$. The property will help us show the improved bounds for paths of length at least $5$, see Remarks~\ref{remark:p2p2} and~\ref{rmk:pathlonger}.

 Here we present the following simple lemmas, since the proofs of these lemmas are similar, for simplicity, we only give the proof of Lemma~\ref{lem:star} in detail. By the first result in Theorem~\ref{thm:alon}, we only consider the cases that $n$ is a sufficiently large number and $k\ge 6$ is an integer.

\begin{lemma}\label{lem:star}
For any set of $n$ $k$-colorings of $K_{n}$ associated with $n$ vertices, there exists a vertex $x\in V(K_{n})$, a set $S$ with $|S|\geqslant\frac{n-1}{k}$ and $f_{x}(xs_{1})=f_{x}(xs_{2})$ for all distinct $s_{1},s_{2}\in S$, and a set $P=V(K_{n})\setminus(S\cup\{x\})$, such that the number of triples $(s,s',p)\in S\times S\times P$ with $f_{p}(xs)=f_{p}(xs')$ is at least $\frac{n^{3}}{24k^{3}}$.
\end{lemma}

\begin{proof}[Proof of Lemma~\ref{lem:star}]
Consider the complete graph $K_{n}$, for any set of $n$ $k$-colorings of $K_{n}$, we can take an arbitrary vertex $x\in V(K_{n})$, by pigeonhole principle, there exists a set $S$ with $|S|\geqslant\frac{n-1}{k}$ such that $f_{x}$ assigns the same color to all edges $xs\in E(K_{n})$ with $s\in S$. Now we set $P:=V(K_{n})\setminus(S\cup\{x\})$ and conut the number of triples $(s,s',p)\in S\times S\times P$ with the property that $f_{p}(xs)=f_{p}(xs')$. Let $f_{p}^{-1}(i)$ be the set of vertices $s\in S$ such that $f_{p}(xs)=i$, as each vertex $p\in P$ contributes
\begin{equation*}
    \sum\limits_{i=1}^{k}\binom{|f_{p}^{-1}(i)|}{2}\geqslant k\cdot\binom{|S|/k}{2}\geqslant\frac{n^{2}}{12k^{3}}
\end{equation*}
many such triples by the convexity of the function $\binom{x}{2}$. Moreover, since $n$ is sufficiently large, we have $|P|\geqslant n-\frac{n-1}{k}-1\geqslant\frac{n}{2}$. Hence the total number of triples $(s,s',p)\in S\times S\times P$ with the desired property is at least $\frac{n^{3}}{24k^{3}}$.
\end{proof}

\begin{lemma}\label{lemma:matching}
For any set of $3n$ $k$-colorings of $K_{3n}$ associated with $3n$ vertices, there is a subset $Y\in V(K_{3n})$ of size $n$ and a matching $M=\{e_{1},e_{2},\ldots,e_{n}\}$ on vertex set $V(K_{3n})\setminus Y$, such that the number of triples $(e,e',p)\in M\times M\times Y$ with $f_{p}(e)=f_{p}(e')$ is at least $\frac{n^{3}}{3k}$.
\end{lemma}

\begin{lemma}\label{lemma:cliques}
For any set of $2n$ $k$-colorings of $K_{2n}$ associated with $2n$ vertices, there are disjoint subsets $A$ and $B$ of size $n$, such that the number of triples $(a,b_{1},b_{2})\in A\times B\times B$ with $f_{a}(ab_{1})=f_{a}(ab_{2})$ is at least $\frac{n^{3}}{3k}$.
\end{lemma}

\begin{lemma}\label{lem:KST}
For any set of $3n$ $k$-colorings of $K_{3n}$ associated with $3n$ vertices, there exists a subset $X$ of size $n$ and a complete graph $K_{2n}$ on vertex set $V(K_{3n})\setminus X$ such that the number of triples $(b,e,e')\in X\times E(K_{2n})\times E(K_{2n})$ with $f_{b}(e)=f_{b}(e')$ is at least $\frac{n^{5}}{3k}$.

\end{lemma}
\section{Dense graphs}\label{sec:dense}
In this section, we mainly focus on dense graphs such as complete graphs and complete bipartite graphs. 

\subsection{Cliques with at least $8$ vertices}\label{sec:clique}
Here we first prove that $C(n,K_{8})=\Omega(n^{{2/3}})$. Let $k:=cn^{{2/3}}$, where the constant $c$ is very small. We consider the complete graph with $2n$ vertices. For any set of $2n$ $k$-colorings of $K_{2n}$, our aim is to find a copy of $K_{8}$ such that none of its vertices can induce a rainbow coloring. By Lemma~\ref{lemma:cliques}, we can partition the vertex set of $K_{2n}$ into two part $A$ and $B$ with $|A|=|B|=n$ and then the number of triples $(a,b_{1},b_{2})\in A\times B\times B$ such that $f_{a}(ab_{1})=f_{a}(ab_{2})$ is at least $\frac{n^{{7/3}}}{3c}$. By pigeonhole principle, there exists a pair of vertices in $B$, called $(b_{1},b_{2})$, such that there are at least $\frac{n^{{1/3}}}{3c}$ many distinct vertices $a\in A$ satisfying $f_{a}(ab_{1})=f_{a}(ab_{2})$. Then we choose a subset $A'$ which consists of the vertices satisfying the above property with size $|A'|=\frac{n^{{1/3}}}{3c}$.

Let $E(A')$ be the edge set of complete graph $K_{|A'|}$, then we use $f_{b_{1}} $ to color the edges in $E(A')$. Note that there are $|E(A')|=\binom{n^{{1/3}}/{3c}}{2}>\frac{n^{{2/3}}}{19c}\gg 300k$ edges in the complete graph induced by $A'$ as the constant $c$ can be very small, hence by pigeonhole principle, we can recursively find a set of disjoint triples of the form $(h_{3i+1},h_{3i+2},h_{3i+3})\in E(A')\times E(A')\times E(A')$ with the property that $f_{b_{1}}(h_{3i+1})=f_{b_{1}}(h_{3i+2})=f_{b_{1}}(h_{3i+3})$ till it covers $99\%$ of the edges in $E(A')$. Moreover, we do the same operation but change $b_{1}$ to $b_{2}$, then we obtain two groups of internally disjoint triples, each of which covers $99\%$ of the edges in $E(A')$. Now we can choose three different edges $e^{1},e^{2},e^{3}$, where $e^{1}$ belongs to both groups of triples, such that $f_{b_{1}}(e^{1})=f_{b_{1}}(e^{2})$ and $f_{b_{2}}(e^{1})=f_{b_{2}}(e^{3})$.
Next we select the vertices of $\{e^{1},e^{2},e^{3}\}$, if there are some common vertices of these edges, we arbitrarily add some vertices from $A'$ to make sure there are $6$ distinct vertices, denoted by $\{a_{1},a_{2},a_{3},a_{4},a_{5},a_{6}\}$. Since all of them belong to the set $A'$, we have $f_{a_{i}}(a_{i}b_{1})=f_{a_{i}}(a_{i}b_{2})$ for $1\leqslant i\leqslant 6$. As a consequence, we find a copy of $K_{8}$ induced by $\{b_{1},b_{2},a_{1},a_{2},a_{3},a_{4},a_{5},a_{6}\}$, which does not admit a rainbow coloring. The proof of $C(n,K_{8})=\Omega(n^{{2/3}})$ is finished.
 
 \begin{remark}
   
 To show $C(n,K_{r})=\Omega(n^{{2/3}})$ for $r>8$, we just need to add $r-8$ vertices from $A'$ to the selected vertex set of $K_{8}$.
 \end{remark}

\subsection{Complete bipartite graph $K_{s,t}$ with $t\geqslant s\geqslant 7$}\label{sec:completebipartite}
Here we first prove that $C(n,K_{7,7})=\Omega(n^{{2/3}})$. Let $k:=cn^{{2/3}}$, where the constant $c$ is very small. We consider a complete graph with $3n$ vertices. For any set of $3n$ $k$-colorings of $K_{3n}$, we try to find a copy of $K_{7,7}$ such that none of its vertices can induce a rainbow coloring. By Lemma~\ref{lem:KST}, we can partition the vertex set of $K_{3n}$ into two parts, say $V_{L}\cup V_{R}$, where $|V_{L}|=2n$ and $|V_{R}|=n$, write $E_{L}$ for the edge set of $K_{|V_{L}|}$, the number of the triples $(b,e,e')\in V_{R}\times E_{L}\times E_{L}$ with the property that $f_{b}(e)=f_{b}(e')$ is at least $\frac{n^{{13/3}}}{3c}$. As the number of edges in $E_{L}$ is $\binom{2n}{2}\leqslant 2n^{2}$, by pigeonhole principle, there exists some pair $(e_{1},e_{2})\in E_{L}\times E_{L}$ such that there are at least $\frac{n^{{13/3}}}{3c}\cdot\frac{1}{4n^{4}}\geqslant\frac{n^{{1/3}}}{12c}$ many vertices $b\in V_{R}$ satisfying $f_{b}(e_{1})=f_{b}(e_{2})$. Now let $V_{R}'$ be a subset of $V_{R}$ consisting of the vertices $b\in V_{R}$ such that $f_{b}(e_{1})=f_{b}(e_{2})$, and $|V_{R}'|\geqslant \frac{n^{{1/3}}}{12c}$. Moreover, we write the vertices of $e_{1}$ and $e_{2}$ as $e_{1}=v_{1}v_{2}$ and $e_{2}=u_{1}u_{2}$, respectively.

Next, we consider the coloring functions $f_{v_{1}}$, $f_{v_{2}}$, $f_{u_{1}}$ and $f_{u_{2}}$ on the edge set $E(V_{R}')$. Here if $e_{1}$ and $e_{2}$ have a common vertex, then we can add an arbitrary vertex from $V_{L}$ and regard it as the vertex $u_{2}$.

Note that $|E(V_{R}')|=\binom{n^{{1/3}}/12c}{2}\geqslant\frac{n^{{2/3}}}{289c^{2}}$, by pigeonhole principle, we can recursively select a family $\mathcal{F}_{1}$ of disjoint pairs of edges $(h_{2i+1},h_{2i+2})\in E(V_{R}')\times E(V_{R}')$ with the property that $f_{v_{1}}(h_{2i+1})=f_{v_{1}}(h_{2i+2})$, till it covers $99\%$ of the edges in $E(V_{R}')$. This is possible since we can set $c$ to be small enough so that $\frac{|E(V_{R}')|}{100}\gg 29k$. Then we repeat the same operations but change $v_{1}$ to $v_{2}$ and change the pairs of edges to the sets of $7$ edges, we can similarly obtain a family $\mathcal{F}_{2}$ of pairwise disjoint sets of $7$ edges, which covers $99\%$ of the edges in $E(V_{R}')$ and each set of $7$ edges receives the same color from the function $f_{v_{2}}$. We continue to do the same operations twice, but replace the corresponding vertices with $u_{1}$ and with $u_{2}$, and also replace the sets of $7$ edges to the sets of $16$ edges and $29$ edges respectively. Finally, we can obtain four families $\mathcal{F}_{1},\mathcal{F}_{2},\mathcal{F}_{3},\mathcal{F}_{4}$ of internally disjoint sets of edges.  

 We can pick one edge such that for each family, there is some set containing this edge, moreover, we denote this edge as $e^{1}$. After we choose the edge $e^{1}$, then the edge $e^{2}$ with $(e^{1},e^{2})\in\mathcal{F}_{1}$ is determined. Next, consider the set of edges in $\mathcal{F}_{2}$ which contains $e^{1}$, we need to carefully pick some edge $e^{3}$ from this set, such that the vertices in $\{e^{1},e^{2},e^{3}\}$ does not form any odd cycle. Indeed, there are $7$ choices but the number of potential edges which can lead to an odd cycle is at most $\binom{4}{2}=6$. Similarly, we need to carefully choose some edge $e^{4}$ from the set in $\mathcal{F}_{3}$ which contains the edge $e^{1}$, such that the vertices $\{e^{1},e^{2},e^{3},e^{4}\}$ does not form any odd cycle. This is also possible as there are $16$ choices but the number of potential edges which may induce an odd cycle are at most $\binom{6}{2}=15$. Also we pick an edge $e^{5}$ from the $29$-element set in $\mathcal{F}_{4}$ which contains $e^{1}$ such that $\{e^{1},e^{2},e^{3},e^{4},e^{5}\}$ does not form any odd cycle for the similar reasons.
 
Finally we can find $7$ edges $e_{1},e_{2},e^{1},e^{2},e^{3},e^{4},e^{5}$ with the following properties.
\begin{itemize}
    \item $f_{v_{1}}(e^{1})=f_{v_{1}}(e^{2})$; $f_{v_{2}}(e^{1})=f_{v_{2}}(e^{3})$; $f_{u_{1}}(e^{1})=f_{u_{1}}(e^{4})$; $f_{u_{2}}(e^{1})=f_{u_{2}}(e^{5})$.
\end{itemize}
If there is some common vertex between any pair of $e^{i}$ and $e^{j}$ for $1\leqslant i\leqslant j\leqslant 5$, then we add some vertices from $V_{R}'$ to guarantee that the edges can be written as $e^{1}=a_{1}a_{2}$, $e^{2}=a_{3}a_{4}$, $e^{3}=a_{5}a_{6}$, $e^{4}=a_{7}a_{8}$ and $e^{5}=a_{9}a_{10}$. Moreover, we need to add some edges from $E(G)$ to form a copy of complete bipartite graph $K_{7,7}$, which contains $\{e_{1},e_{2},e^{1},e^{2},e^{3},e^{4},e^{5}\}$
and its corresponding vertex set $V(K_{7,7})=\{v_{1},v_{2},u_{1},u_{2},a_{1},a_{2},a_{3},a_{4},a_{5},a_{6},a_{7},a_{8},a_{9},a_{10}\}$.
Note that, for $1\leqslant i\leqslant 10$, we have $f_{a_{i}}(e_{1})=f_{a_{i}}(e_{2})$ by the construction of $V_{R}'$. Hence, we find a copy of $K_{7,7}$ such that none of its vertices can induce a rainbow coloring. The proof of $C(n,K_{7,7})=\Omega(n^{{2/3}})$ is finished. 

\begin{remark}

To show $C(n,K_{s,t})=\Omega(n^{{2/3}})$ for $t\geqslant s\geqslant 7$ and $t>7$, we just need to add $s+t-14$ vertices from $V_{R}'$ to the selected vertex set of $K_{7,7}$.
Moreover, we actually prove that $C(n,P_{3}\sqcup K_{5,5}\sqcup\overline{K_{h}} )=\Omega(n^{{2/3}})$ for arbitrary constant $h$.
\end{remark}

\section{Sparse Graphs}\label{sec:Sparse}
In this section, we prove the improved lower bounds on $C(n,H)$ when $H$ are relatively sparse graphs. 
\subsection{Paths of short length}\label{sec:P4}
First we prove $C(n,P_{4})=\Omega(n^{{1/5}})$. Let $k:=cn^{{1/5}}$, where the constant $c>0$ can be taken sufficiently small. For any set of $2n$ $k$-colorings of $K_{2n}$, denoted by $f$, it suffices to find a copy of $P_{4}$ such that there is no vertex that assigns distinct colors to all edges of this special copy of $P_{4}$.

We partition the vertex set of $K_{2n}$ into two parts $A$ and $B$ with $|A|=|B|=n$. Consider the number of quadruples $(a,b_{1},b_{2},b_{3})\in A\times B\times B\times B$ such that $f_{a}(ab_{1})=f_{a}(ab_{2})=f_{a}(ab_{3})$. Let $f_{a}^{-1}(i)$ be the set of vertices $b\in B$ such that $f_{a}(ab)=i$. Observe that each element in $A$ contributes
\begin{equation*}
    \sum\limits_{i=1}^{cn^{\frac{1}{5}}}\binom{|f_{a}^{-1}(i)|}{3}\geqslant cn^{\frac{1}{5}}\cdot\binom{|B|/cn^{\frac{1}{5}}}{3}\geqslant\frac{n^{\frac{13}{5}}}{7c^2}
\end{equation*}
many such quadruples by the convexity of the function $\binom{x}{3}$.
Moreover, since there are $n$ choices for $a\in A$, totally there are at least $\frac{n^{{18/5}}}{7c^2}$ such quadruples in $A\times B\times B\times B$. By pigeonhole principle, there is a triple of vertices in $B$, called $(b_{1},b_{2},b_{3})$, such that there are at least $\frac{n^{{3/5}}}{7c^2}$ many vertices $a\in A$ satisfying $f_{a}(ab_{1})=f_{a}(ab_{2})=f_{a}(ab_{3})$. Next let $A'\subseteq A$ be a subset with $|A'|\geqslant\frac{n^{{3/5}}}{7c^2}$, which consists of the vertices satisfying the above property.
Then we consider the $k^3$-coloring $(f_{b_{1}}(b_{2}a),f_{b_{2}}(b_{2}a),f_{b_{3}}(b_{2}a))$ for all $a\in A'$. By pigeonhole principle, there are at least $\frac{1}{7c^{5}}\ge 2 $ elements in $A'$ receiving the same color, since $c$ is sufficiently small. We choose two of them arbitrarily, and denote them as $a_{1}$ and $a_{2}$. Now through the above analysis, we can find a copy of $P_{4}$ on vertex set $\left \{b_{1},a_{1},b_{2},a_{2},b_{3}\right \} $ with the following properties:
\begin{itemize}
    \item $f_{a_{1}}(a_{1}b_{1})=f_{a_{1}}(a_{1}b_{2})$; $f_{a_{2}}(a_{2}b_{3})=f_{a_{2}}(a_{2}b_{2})$;
    \item $f_{b_{1}}(b_{2}a_{2})=f_{b_{1}}(b_{2}a_{1})$; $f_{b_{2}}(b_{2}a_{2})=f_{b_{2}}(b_{2}a_{1})$;  $f_{b_{3}}(b_{2}a_{2})=f_{b_{3}}(b_{2}a_{1})$.
\end{itemize}
That means, none of the vertices in this special copy of $P_{4}$ can induce a rainbow coloring. The proof of $C(n,P_{4})=\Omega(n^{{1/5}})$ is finished.

\subsection{Stars}\label{sec:stars}
For the star $S_{t}$ with $t$ leaves, we first consider the case of $t=4$ and then obtain a better bound for any larger positive integer $t\geqslant 5$.

\subsubsection{Star with $4$ leaves}
Here we give a proof of $C(n,S_{4})=\Omega(n^{{1/3}})$. Let $k:=cn^{{1/3}}$, and $c$ be sufficiently small. Consider the complete graph $K_{n}$, for any set of $n$ $k$-colorings of $K_{n}$, we need to find a copy of $S_{4}$ such that no vertex assigns distinct colors to all edges of this special copy of $S_{4}$. By Lemma~\ref{lem:star}, there exists a vertex $x\in V(K_{n})$, a set $S$ with $|S|\geqslant{(n-1)/k}$ and $f_{x}(xs_{1})=f_{x}(xs_{2})$ for all distinct $s_{1},s_{2}\in S$ and a set $P=V(K_{n})\setminus(S\cup\{x\})$, such that the number of triples $(s,s',p)\in S\times S\times P$ with $f_{p}(xs)=f_{p}(xs')$ is at least ${n^{2}/24c^{3}}$.

Then by pigeonhole principle, there exists a pair of elements $(s,s')\in S\times S$ such that the number of vertices $p\in P$ with $f_{p}(xs)=f_{p}(xs')$ is at least $({n^{2}/24c^{3}})\cdot\frac{c^{2}}{n^{{4/3}}}=\frac{n^{{2/3}}}{24c}$. Let $A$ be the subset consisting of the above vertices $p\in P$.

For the coloring function $f_{s}$, by pigeonhole principle, there is a subset $A'\subseteq A$ with $|A'|\geqslant\frac{n^{{2/3}}}{24c}\cdot\frac{1}{cn^{{1/3}}}=\frac{n^{{1/3}}}{24c^{2}}$, such that $f_{s}$ assigns the same color to all edges $xa$ with $a\in A'$. Similarly for $f_{s'}$, by pigeonhole principle again, there is a subset $A''\subseteq A'$ with $|A''|\geqslant\frac{n^{{1/3}}}{24c^{2}}\cdot\frac{1}{cn^{{1/3}}}={1/(24c^{3})}$, such that $f_{s'}$ assigns the same color to all edges $xa$ with $a\in A''$. Since $c$ is small enough, we have ${1/(24c^{3})}\ge 2$. Then we can find a pair of distinct elements $(a,a')\in A''\times A''$.

Finally, we can find a copy of $S_{4}$ on vertex set $\{x,a,a',s,s'\}$, where $x$ is the center, satisfies the following properties:
\begin{itemize}
    \item $f_{x}(xs)=f_{x}(xs')$; $f_{s}(xa)=f_{s}(xa')$; $f_{s'}(xa)=f_{s'}(xa')$; $f_{a}(xs)=f_{a}(xs')$; $f_{a'}(xs)=f_{a'}(xs')$.
\end{itemize}
Hence, we find a copy of $S_{4}$, such that none of its vertices assigns distinct colors to all edges. The proof of  $C(n,S_{4})=\Omega(n^{{1/3}})$ is finished.

\begin{remark}\label{remark:p2k2k2}
  In the above proof, after we have found a set $P$ with $|P|\geqslant\frac{n^{{2/3}}}{24c}$ and the triple $\{x,s,s'\}\in P\times S\times S$ satisfying $f_{x}(xs)=f_{x}(xs')$, $f_{p}(xs)=f_{p}(xs')$ for all $p\in P$, we can find a matching $M$ on the vertex set $P$ and consider the $k^2$-coloring $(f_{s},f_{s'})$. Note that $|P|\geqslant\frac{n^{{2/3}}}{24c}\gg  k^2$, by pigeonhole principle, we can find a pair of edges $e_{1}$ and $e_{2}$ in the matching $M$ such that $f_{s}(e_{1})=f_{s'}(e_{1})$ and $f_{s}(e_{2})=f_{s'}(e_{2})$. This gives a copy of $P_{2}\sqcup K_{2}\sqcup K_{2}$ with edge set $\{xs,xs',e_{1},e_{2}\}$ such that no vertex induces a rainbow coloring. Thus we have $C(n,P_{2}\sqcup K_{2}\sqcup K_{2})=\Omega(n^{{1/3}})$.
\end{remark}

\begin{remark}\label{remark:p2p2}
  We can pick a vertex $p\in P$, and use the $k^2$-coloring $(f_{s},f_{s'})$ to color all edges incident to vertex $p$ in the graph induced by $P$. As $|P|\geqslant\frac{n^{{2/3}}}{24c}\gg  k^2$, we can find a pair of vertices $a,b\in P$ such that the edges $pa$ and $pb$ receive the same color under the $k^2$-coloring $(f_{s},f_{s'})$. This gives a copy of $P_{2}\sqcup P_{2}$ with edges set $\{pa,pb,xs,xs'\}$ such that no vertex induces a rainbow coloring. Thus we have $C(n,P_{2}\sqcup P_{2})=\Omega(n^{{1/3}})$. Moreover, we can add arbitrary $r$ isolated vertices from the set $P$ to obtain same lower bound for $C(n,P_{2}\sqcup K_{2}\sqcup K_{2}\sqcup rK_{1})$ and $C(n,P_{2}\sqcup P_{2}\sqcup rK_{1})$ for arbitrary non-negative integer $r$. Moreover, since any path of length at least $5$ contains a copy of $P_{2}\sqcup P_{2}$, the second result in Theorem~\ref{thm:paths} follows.
\end{remark}

\subsubsection{Stars with more than $4$ leaves}\label{star more}
We next consider the case of $t=5$. Let $k:=cn^{{1/2}}$ with sufficiently small constant $c>0$. For any set of $n$ $k$-colorings of $K_{n}$, by Lemma~\ref{lem:star}, we can find a vertex $x\in V(K_{n})$, a set $S$ of size $|S|\geqslant{(n-1)/k}$ with $f_{x}(xs_{1})=f_{x}(xs_{2})$ for all distinct $s_{1},s_{2}\in S$ and a subset $P=V(K_{n})\setminus(S\cup\{x\})$, such that the number of triples $(s,s',p)\in S\times S\times P$ with $f_{p}(xs)=f_{p}(xs')$ is at least $\frac{n^{{3/2}}}{24c^{3}}$. Also using pigeonhole principle, we can find a pair of elements $(s,s')\in S\times S$ such that the number of vertices $p\in P$ with $f_{p}(xs)=f_{p}(xs')$ is at least $\frac{n^{{3/2}}}{24c^{3}}\cdot{c^{2}/n}=\frac{{n^{{1/2}}}}{24c}$. Let $A$ consist of the above vertices $p\in P$, without loss of generality, we assume that $|A|$ is divided by $9$.

Note that $c$ is small enough, hence for $i=1,2,\ldots,\frac{2}{9}\cdot |A|$, by pigeonhole principle, we can
recursively find disjoint triples $(p_{3i+1},p_{3i+2},p_{3i+3})\in A\times A\times A$ with the property that $f_{s}(xp_{3i+1})=f_{s}(xp_{3i+2})=f_{s}(xp_{3i+3})$. Moreover, we do the same operations for another vertex $s'$, and then we will obtain two sets of internally disjoint triples. Since the total number of vertices we obtain is larger than $|A|$, we can find one triple from each set respectively, such that they intersect, that is, there are distinct triples $\left \{ p_{t},p_{j},p_{k} \right \}\subseteq A$ and $\left \{ p_{t},p_{h},p_{m} \right \}\subseteq A $ such that $f_{s}(xp_{t})=f_{s}(xp_{j})=f_{s}(xp_{k})$, $f_{s'}(xp_{t})=f_{s'}(xp_{h})=f_{s'}(xp_{m})$, where $j=h$ or $k=m$ is also allowed. Finally, we can select a set of vertices $\left \{ x,s,s',p_{t},p_{h},p_{k} \right \}$ to form a copy of star centered at vertex $x$ with $5$ leaves with the properties:
\begin{itemize}
    \item $f_{x}(xs)=f_{x}(xs')$; $f_{s}(xp_{t})=f_{s}(xp_{k})$; $f_{s'}(xp_{t})=f_{s'}(xp_{h})$;
    \item $f_{p_{t}}(xs)=f_{p_{t}}(xs')$; $f_{p_{h}}(xs)=f_{p_{h}}(xs')$; $f_{p_{k}}(xs)=f_{p_{k}}(xs')$.
\end{itemize}

 None of the vertices assigns distinct colors to all edges of this $S_{5}$. The proof of $C(n,S_{5})=\Omega(n^{{1/2}})$ is finished.

\begin{remark}
  
For the remaining cases of $t>5$, we can just take $t-5$ vertices from $A$ and add them to the selected set $\left \{ x,s,s',p_{t},p_{h},p_{k} \right \}$ to obtain a copy of $S_{t}$ such that none of its vertices assigns distinct colors to all its edges, we omit the details here.

\end{remark}

\subsection{Matchings}\label{sec:matchings}
For the matching $I_{t}$, we first consider the cases of $t=4$ and $t\in\{5,6\}$, respectively. Furthermore, for any larger positive integer $t\geqslant 7$, we can even obtain a better lower bound.

\subsubsection{Matching of size $4$}\label{subsec:I4}
First we prove that $C(n,I_{4})=\Omega(n^{{1/5}})$. Let $k:=cn^{{1/5}}$ with $\frac{1}{c^{5}}\gg 12$. We consider the complete graph with $3n$ vertices. For any set of $3n$ $k$-colorings of $K_{3n}$, it suffices to find a copy of $I_{4}$ such that there is no vertex that assigns distinct colors to all edges of this special copy of $I_{4}$. By Lemma~\ref{lemma:matching}, we can partition the vertex set of $K_{3n}$ into two parts, say $X\cup Y$, where $|X|=2n$ and $|Y|=n$ and there is a matching $M=\{e_{1},e_{2},\ldots,e_{n}\}$ in $X$, where $e_{i}$ and $e_{j}$ are pairwise disjoint for all $1\leqslant i<j\leqslant n$ such that the number of triples $(e,e',p)\in M\times M\times Y$ with the property that $f_{p}(e)=f_{p}(e')$ is at least $\frac{n^{{14/5}}}{3c}$. By pigeonhole principle, there is a pair of edges in $M$, called $(e_{1},e_{2})$, such that there are at least $\frac{n^{{4/5}}}{3c}$ many vertices $p\in Y$ satisfying $f_{p}(e_{1})=f_{p}(e_{2})$. Let $A\subseteq Y$ consist of all the vertices $p\in Y$ such that $f_{p}(e_{1})=f_{p}(e_{2})$, as we have mentioned above, $|A|\geqslant \frac{n^{{4/5}}}{3c}$. Write the vertices of $e_{1}$ and $e_{2}$ as $e_{1}:=v_{1}u_{1}$, $e_{2}:=v_{2}u_{2}$ respectively. As the size of $A$ is at least $|A|\geqslant\frac{n^{{4/5}}}{3c}$, we can choose an arbitrary copy of matching $I_{{|A|/2}}$ in $A$, whose edge set is $H=\{h_{1},h_{2},\ldots,h_{{|A|/2}}\}$. Now we consider the coloring functions $f_{v_{1}}$, $f_{u_{1}}$, $f_{v_{2}}$ and $f_{u_{2}}$ of edges in $H$. Using pigeonhole principle for four times, we can find a subset $H'\subseteq H$ of edges with size $|H'|\geqslant\frac{n^{{4/5}}}{6c}\cdot\frac{1}{(cn^{{1/5}})^{4}}=\frac{1}{6c^{5}}\ge 2$ such that there is a pair of edges $h_{1}$ and $h_{2}$ in $H'$ satisfying $f_{v_{1}}(h_{1})=f_{v_{1}}(h_{2})$, $f_{v_{2}}(h_{1})=f_{v_{2}}(h_{2})$, $f_{u_{1}}(h_{1})=f_{u_{1}}(h_{2})$ and $f_{u_{2}}(h_{1})=f_{u_{2}}(h_{2})$. Fix such a pair of edges $h_{1}$ and $h_{2}$ and write them as $h_{1}:=a_{1}b_{1}$ and $h_{2}:=a_{2}b_{2}$, respectively. By the above analysis, we can find a copy of $I_{4}$ in $K_{3n}$ whose edge set is $\{e_{1},e_{2},h_{1},h_{2}\}$ with the following properties:
\begin{itemize}
    \item $f_{v_{1}}(h_{1})=f_{v_{1}}(h_{2})$; $f_{v_{2}}(h_{1})=f_{v_{2}}(h_{2})$; $f_{u_{1}}(h_{1})=f_{u_{1}}(h_{2})$; $f_{u_{2}}(h_{1})=f_{u_{2}}(h_{2})$;
    \item $f_{a_{1}}(e_{1})=f_{a_{1}}(e_{2})$; $f_{a_{2}}(e_{1})=f_{a_{2}}(e_{2})$; $f_{b_{1}}(e_{1})=f_{b_{1}}(e_{2})$; $f_{b_{2}}(e_{1})=f_{b_{2}}(e_{2})$.
\end{itemize}
The proof of $C(n,I_{4})=\Omega(n^{1/5})$ is finished since none of the vertices in this $I_{4}$ we find  can induce a rainbow coloring.

\subsection{Matchings of sizes $5$ and $6$}\label{subsec:I5}
In this part we show the better bounds for matchings of sizes $5$ and $6$. More precisely, we will show that $C(n,I_{5})=\Omega(n^{{1/3}})$ in details. The proof of $C(n,I_{6})=\Omega(n^{{1/3}})$ is similar, so we just provide a simple remark.

Let $k:=cn^{{1/3}}$, where the constant $c$ is very small. We also focus on the complete graph with $3n$ vertices. For any set of $3n$ $k$-colorings of $K_{3n}$, we need to show that there is a copy of $I_{5}$ such that none of its vertices induces a rainbow coloring. By Lemma~\ref{lemma:matching}, there exists a partition of the vertex set of $K_{3n}$ into $X\cup Y$ with $|X|=2n$ and $|Y|=n$ and a matching $M=\{e_{1},e_{2},\ldots,e_{n}\}$ in $X$, such that the number of triples $(e,e',p)\in M\times M\times Y$ with the property that $f_{p}(e)=f_{p}(e')$ is at least $\frac{n^{{8/3}}}{3c}$. Using pigeonhole principle, we can find a pair of edges in $M:=(e_{1},e_{2})$ with $e_{1}=v_{1}u_{1}$ and $e_{2}=v_{2}u_{2}$, such that there are at least $\frac{n^{{2/3}}}{3c}$ many vertices $p\in Y$ satisfying $f_{p}(e_{1})=f_{p}(e_{2})$. Then we denote $A$ as the subset of $Y$ consisting of the vertices $p\in Y$ such that $f_{p}(e_{1})=f_{p}(e_{2})$ and $|A|\geqslant\frac{n^{{2/3}}}{3c}$. As $|A|=\frac{n^{{2/3}}}{3c}$, we can choose an arbitrary copy of matching $I_{|A|/2}$ in $A$, whose edge set is $H=\{h_{1},h_{2},\ldots,h_{|A|/2}\}$.

Now we consider the $k^2$-coloring function $(f_{v_{1}}, f_{u_{1}})$ of edges in $H$. By pigeonhole principle, we can recursively find disjoint triples $(h_{3i+1},h_{3i+2},h_{3i+3})\in H\times H\times H$ with the property $f_{v_{1}}(h_{3i+1})=f_{v_{1}}(h_{3i+2})=f_{v_{1}}(h_{3i+3})$ and $f_{u_{1}}(h_{3i+1})=f_{u_{1}}(h_{3i+2})=f_{u_{1}}(h_{3i+3})$ till they cover $99\%$ of the elements in $H$ (this is possible because $c$ is small enough such that $\frac{1}{100}|H|\geqslant 3k^2$), then we do same operation but change $v_{1}$ to $v_{2}$ and $u_{1}$ to $u_{2}$ respectively, we will also obtain many disjoint triples cover $99\%$ of the elements in $H$ and every such chosen triple is colored same by the $k^2$-coloring $(f_{v_{2}}, f_{u_{2}})$. Then we can find two triples from the set of triples such that they intersect, namely, we can find two triples $\left \{ h_{t},h_{j},h_{k} \right \}\subseteq H$ and $\left \{ h_{t},h_{q},h_{m} \right \}\subseteq H $ such that $f_{v_{1}}(h_{t})=f_{v_{1}}(h_{j})=f_{v_{1}}(h_{k})$, $f_{u_{1}}(h_{t})=f_{u_{1}}(h_{j})=f_{u_{1}}(h_{k})$, $f_{u_{2}}(h_{t})=f_{u_{2}}(h_{q})=f_{u_{2}}(h_{m})$ and $f_{v_{2}}(h_{t})=f_{v_{2}}(h_{q})=f_{v_{2}}(h_{m})$, where $j=q$ or $k=m$ is allowed. Then we can find five edges $\left \{e_{1},e_{2},h_{t},h_{j},h_{m}\right \} $ and denote the vertices of $h_{t},h_{j},h_{m}$ as $\{a_{t},b_{t}\}$, $\{a_{j},b_{j}\}$, $\{a_{m},b_{m}\}$, respectively. Through the above argument, we can see the following properties hold.

\begin{itemize}
    \item $f_{v_{1}}(h_{t})=f_{v_{1}}(h_{j})$; $f_{v_{2}}(h_{t})=f_{v_{2}}(h_{m})$; $f_{u_{1}}(h_{t})=f_{u_{1}}(h_{j})$; $f_{u_{2}}(h_{t})=f_{u_{2}}(h_{m})$;
    \item $f_{a_{t}}(e_{1})=f_{a_{t}}(e_{2})$; $f_{a_{j}}(e_{1})=f_{a_{j}}(e_{2})$; $f_{a_{m}}(e_{1})=f_{a_{m}}(e_{2});$
    \item $f_{b_{t}}(e_{1})=f_{b_{t}}(e_{2})$; $f_{b_{j}}(e_{1})=f_{b_{j}}(e_{2})$; $f_{b_{m}}(e_{1})=f_{b_{m}}(e_{2})$.
\end{itemize}

That means, none of the vertices in this $I_{5}$ we find above, can induce a rainbow coloring. Thus we have $C(n,I_{5})=\Omega(n^{{1/3}})$. The proof is finished.

\begin{remark}

If we want to find a special copy of $I_{6}$ which does not admit the rainbow coloring, we just need to add one edge from $H$ to the selected set $\left \{e_{1},e_{2},h_{t},h_{j},h_{m}\right \}$. We omit the details here.
\end{remark}

\subsubsection{Matchings with larger size}
We mainly prove that $C(n,I_{7})=\Omega(n^{{1/2}})$ here. Let $k:=cn^{{1/2}}$, where the constant $c$ is chosen to be very small. We still consider the complete graph with $3n$ vertices. For any set of $3n$ $k$-colorings of $K_{3n}$. Using the similar argument as in Sections~\ref{subsec:I4} and~\ref{subsec:I5}, there is a partition of the vertex set of $K_{3n}$ into $X\cup Y$ with $|X|=2n$ and $|Y|=n$ and a matching $M=\{e_{1},e_{2},\ldots,e_{n}\}$ in $X$, such that the number of triples $(e,e',p)\in M\times M\times Y$ with the property that $f_{p}(e)=f_{p}(e')$ is at least $\frac{n^{{5/2}}}{3c}$. We can also find a pair of edges in $M:=(e_{1},e_{2})$, such that there are at least $\frac{n^{{1/2}}}{3c}$ many vertices $p\in Y$ satisfying $f_{p}(e_{1})=f_{p}(e_{2})$, where $e_{1}=v_{1}u_{1}$, $e_{2}=v_{2}u_{2}$. Let $A$ be a subset of $Y$ consisting of the vertices $p\in Y$ such that $f_{p}(e_{1})=f_{p}(e_{2})$ and $|A|=\frac{n^{{1/2}}}{3c}$, we can form an arbitrary copy of matching $I_{|A|/2}$ in $A$, whose edge set is $H=\{h_{1},h_{2},\ldots,h_{|A|/2}\}$.

Now we consider the $k$-coloring function $f_{v_{1}}$ of edges in $H$. By pigeonhole principle, we can recursively find a set of disjoint quintuples of the form $(h_{5i+1},h_{5i+2},h_{5i+3},h_{5i+4},h_{5i+5})\in H\times H\times H\times H\times H$ with the property that $f_{v_{1}}(h_{5i+1})=f_{v_{1}}(h_{5i+2})=f_{v_{1}}(h_{5i+3})=f_{v_{1}}(h_{5i+4})=f_{v_{1}}(h_{5i+5})$ till it covers $99\%$ of the elements in $H$ (this is possible because we can set $c$ to be small enough such that ${1/100}|H|\gg 5k$). Then we do same operations but change $v_{1}$ to $u_{1}$, we can also obtain another set of disjoint quintuples, which covers $99\%$ of the elements in $H$. We continue the same operations twice, but replacing the corresponding vertices with $u_{2}$ and $v_{2}$. As a consequence, we can obtain four sets of internally disjoint quintuples, all of which have the desired properties. Next, we pick one quintuple from each set respectively, such that they contain some common element. That means there are $\left \{ h_{t},h_{k_{1}},h_{k_{2}},h_{k_{3}},h_{k_{4}} \right \}\subseteq H$, $\left \{ h_{t},h_{k_{5}},h_{k_{6}},h_{k_{7}},h_{k_{8}} \right \}\subseteq H$, $\left \{ h_{t},h_{k_{9}},h_{k_{10}},h_{k_{11}},h_{k_{12}} \right \}\subseteq H$ and $\left \{ h_{t},h_{k_{13}},h_{k_{14}},h_{k_{15}},h_{k_{16}} \right \}\subseteq H$ such that $f_{v_{1}}(h_{t})=f_{v_{1}}(h_{k_{1}})=f_{v_{1}}(h_{k_{2}})=f_{v_{1}}(h_{k_{3}})=f_{v_{1}}(h_{k_{4}})$, $f_{u_{1}}(h_{t})=f_{u_{1}}(h_{k_{5}})=f_{u_{1}}(h_{k_{6}})=f_{u_{1}}(h_{k_{7}})=f_{u_{1}}(h_{k_{8}})$, $f_{v_{2}}(h_{t})=f_{v_{2}}(h_{k_{9}})=f_{v_{2}}(h_{k_{10}})=f_{v_{2}}(h_{k_{11}})=f_{v_{2}}(h_{k_{12}})$ and  $f_{u_{2}}(h_{t})=f_{u_{2}}(h_{k_{13}})=f_{u_{2}}(h_{k_{14}})=f_{u_{2}}(h_{k_{15}})=f_{u_{2}}(h_{k_{16}})$. Moreover, we can find one edge form each set of the above four quintuples such that they are pairwise disjoint, without loss of generality, let $\left\{ h_{k_{1}},h_{k_{6}},h_{k_{11}},h_{k_{16}} \right\}$ be the set of edges with the disjoint property. Then we pick seven edges $\left \{e_{1},e_{2},h_{t},h_{k_{1}},h_{k_{6}},h_{k_{11}},h_{k_{16}}\right \} $ and denote the vertices of $h_{t},h_{k_{1}},h_{k_{6}},h_{k_{11}},h_{k_{16}}$ as $\{a_{t},b_{t}\}$, $\{a_{k_{1}},b_{k_{1}}\}$, $\{a_{k_{6}},b_{k_{6}}\}$, $\{a_{k_{11}},b_{k_{11}}\}$ and $\{a_{k_{16}},b_{k_{16}}\}$, respectively. Then we can see that the following properties hold.

\begin{itemize}
    \item $f_{v_{1}}(h_{t})=f_{v_{1}}(h_{{k_{1}}})$; $f_{v_{2}}(h_{t})=f_{v_{2}}(h_{{k_{9}}})$; $f_{u_{1}}(h_{t})=f_{u_{1}}(h_{{k_{5}}})$; $f_{u_{2}}(h_{t})=f_{u_{2}}(h_{{k_{16}}})$;
    \item $f_{a_{t}}(e_{1})=f_{a_{t}}(e_{2})$; $f_{b_{t}}(e_{1})=f_{b_{t}}(e_{2})$;
    \item $f_{a_{{k_{1}}}}(e_{1})=f_{a_{{k_{1}}}}(e_{2})$; $f_{a_{{k_{5}}}}(e_{1})=f_{a_{{k_{5}}}}(e_{2})$; $f_{a_{{k_{9}}}}(e_{1})=f_{a_{{k_{9}}}}(e_{2})$; $f_{a_{{k_{16}}}}(e_{1})=f_{a_{{k_{16}}}}(e_{2})$;
    \item  $f_{b_{{k_{1}}}}(e_{1})=f_{b_{{k_{1}}}}(e_{2})$; $f_{b_{{k_{5}}}}(e_{1})=f_{b_{{k_{5}}}}(e_{2})$; $f_{b_{{k_{9}}}}(e_{1})=f_{b_{{k_{9}}}}(e_{2})$; $f_{b_{{k_{16}}}}(e_{1})=f_{b_{{k_{16}}}}(e_{2})$.
\end{itemize}

That means, none of the vertices in this $I_{7}$ we find above, can give a rainbow coloring. Thus the proof of $C(n,I_{7})=\Omega(n^{{1/2}})$ is finished.

\begin{remark}
  
when $t>7$, if we want to find an $I_{t}$ that does not admit a rainbow coloring, we just need to add $t-7$ edges from $H$ to the selected set $\left \{e_{1},e_{2},h_{t},h_{k_{1}},h_{k_{6}},h_{k_{11}},h_{k_{16}}\right \} $. 
\end{remark}
\begin{remark}\label{rmk:pathlonger}
  The improved lower bounds for $C(n,P_{2t-1})$ and $C(n,P_{2t})$ can be obtained from the lower bounds for $C(n,I_{t})$, this is because if $H'\subseteq H$ is a subgraph of $H$ on the same set of vertices, then every lower bound for $C(n,H')$ implies the same lower bound for $C(n,H)$. This already shows that $C(n,P_{t})=\Omega(n^{1/2})$ for $t\geqslant 13$. Actually we can slightly improve this result further via combining the ideas in the proof of $C(n,I_{7})=\Omega(n^{{1/2}})$ and in Remark~\ref{remark:p2k2k2} to show that $C(n,P_{2}\sqcup K_{2}\sqcup K_{2}\sqcup K_{2})=\Omega(n^{{1/2}})$. As the proof is very similar, we omit the details here. Note that any path of length at least $8$ contains $P_{2}\sqcup K_{2}\sqcup K_{2}\sqcup K_{2}$, hence we have $C(n,P_{t})=\Omega(n^{1/2})$ for $t\geqslant 8$. The third result in Theorem~\ref{thm:paths} follows.
\end{remark}

\section{Graphs with at least $6$ edges have polynomial lower bounds }\label{sec:C4}

In this section, we will prove that for any graph $H$ with at least $6$ edges, there exists some constant $b=b(H)>0$ such that $C(n,H)=\Omega(n^{b})$. First, we need the following auxiliary lemma.

\begin{lemma}\label{lemma:666}
If a graph $H$ has at least $6$ edges, then $H$ must contain at least one member of the family $\mathcal{H}_{6}$ as a subgraph, where $\mathcal{H}_{6}=\{C_{4}$, $P_{4}$, $I_{4}$, $S_{4}$, $P_{2}\sqcup P_{2}$, $P_{2}\sqcup K_{2}\sqcup K_{2}\}$.
\end{lemma}
\begin{proof}
It suffices to prove the lemma when $H$ has exactly $6$ edges. Without loss of generality, we can assume that $H$ has no isolated vertex. Our proof is based on the number of connected components of $H$. 
\begin{enumerate}
    \item If $H$ has more than $4$ connected components, then it must contain a copy of $I_{4}$.
    \item If $H$ has $3$ connected components, then by pigeonhole principle, there exists some component with at least $2$ edges, then $H$ must contain a copy of $P_{2}\sqcup K_{2}\sqcup K_{2}$.
    \item If $H$ has $2$ connected components, suppose both components have at least $2$ edges, then $H$ must contain a copy of $P_{2}\sqcup P_{2}$. Next we assume that some component $H_{1}$ has exactly $5$ edges. If $H_{1}$ contains a triangle, then $H_{1}$ contains either a copy of $P_{2}\sqcup K_{2}$ or a copy of $C_{4}$, which means $H$ contains either a copy of $P_{2}\sqcup K_{2}\sqcup K_{2}$ or $C_{4}$. Then we consider the case that $H_{1}$ is a tree, if $H_{1}$ does not contain $S_{4}$ or $P_{4}$, then it is easy to check that $H_{1}$ contains a copy of $P_{2}\sqcup K_{2}$, thus $H$ contains a copy of $P_{2}\sqcup K_{2}\sqcup K_{2}$.
    \item If $H$ has only one connected component, suppose $H$ does not contain a cycle, then either $H$ contains a copy of $S_{4}$ or the longest path in $H$ has at least $4$ edges. Then we consider the case that $H$ contains at least one cycle. Note that if $H$ does not contain a copy of $C_{4}$ or $P_{4}$, then $H$ contains a triangle. In this case, if the remaining three edges are incident to the same vertex on the triangle, then $H$ contains a copy $S_{4}$, otherwise, $H$ will contain a copy of $P_{4}$.
    
\end{enumerate}
\end{proof}
It remains to show that $C(n,C_{4})$ also has the polynomial lower bound, we prove this result as follows. 

\begin{proposition}\label{prop:c4}
\begin{equation*}
    C(n,C_{4})=\Omega(n^{1/3}).
\end{equation*}
\end{proposition}
\begin{proof}[Proof of Proposition~\ref{prop:c4}]
 Let $k:=cn^{{1/3}}$, where the constant $c$ is very small. We consider a complete graph with $2n$ vertices. For any set of $2n$ $k$-colorings of $K_{2n}$, we aim to find a copy of $C_{4}$ such that none of its vertices can induce a rainbow coloring.

By Lemma~\ref{lemma:cliques}, we can partition the vertex set of $K_{2n}$ into two part $A$ and $B$ with $|A|=|B|=n$ and the number of triples $(a,b_{1},b_{2})\in A\times B\times B$ such that $f_{a}(ab_{1})=f_{a}(ab_{2})$ is at least $\frac{n^{{8/3}}}{3c}$. By pigeonhole principle, there is a pair of vertices in $B$, called $(b_{1},b_{2})$, such that there are at least $\frac{n^{{2/3}}}{3c}$  many distinct vertices $a\in A$ satisfying $f_{a}(ab_{1})=f_{a}(ab_{2})$. Then we choose a subset $A'$ which consists of the vertices satisfying the above property. It is obvious that $|A'|\ge \frac{n^{{2/3}}}{3c}$. Consider the $k^2$-coloring $(f_{b_{1}}(b_{1}a),f_{b_{2}}(b_{2}a))$ for all $a\in A'$, by pigeonhole principle, there are $\frac{1}{3c^{3}}\gg 2$ elements in $A'$ receive the same color. Then we pick two vertices and write them as $a_{1}$ and $a_{2}$. Now we find a copy of $C_{4}$ with vertex set $\left \{a_{1}, b_{2}, a_{2}, b_{1}\right \}$ with the following properties 
\begin{itemize}
    \item $f_{a_{1}}(a_{1}b_{1})=f_{a_{1}}(a_{1}b_{2})$; $f_{a_{2}}(a_{2}b_{1})=f_{a_{2}}(a_{2}b_{2})$; $f_{b_{1}}(b_{1}a_{2})=f_{b_{1}}(b_{1}a_{1})$; $f_{b_{2}}(b_{2}a_{2})=f_{b_{2}}(b_{2}a_{1})$.
\end{itemize}
That means, none of the vertices in this $C_{4}$ we find here, can induce a rainbow coloring. The proof of $C(n,C_{4})=\Omega(n^{{1/3}})$ is finished.

\end{proof}
The proofs of Theorems~\ref{thm:paths},~\ref{thm:stars},~\ref{thm:matchings}, Proposition~\ref{prop:c4}, and Remarks~\ref{remark:p2k2k2} and~\ref{remark:p2p2} together give the proof of Theorem~\ref{thm:666}, as well as we can easily deal with the isolated vertices in these cases.

\section{Concluding remarks and open problems}\label{sec:conclusions}

One of the most interesting problems in this topic proposed by Alon and Ben-Eliezer~\cite{2011AlonJC} was to decide whether the order of $C(n,P_{t})$ is polynomial in $n$ with $t\in\{3,4,5,6\}$. In this paper, we show that $C(n,P_{4})=\Omega(n^{{1/5}})$, and provide better bounds for $P_{5}$ and $P_{6}$. Now the case of $P_{3}$ remains open, through we cannot answer this question on $P_{3}$, we give the following polynomial lower bound $C(n,C_{4})=\Omega(n^{{1/3}})$, which perhaps gives some evidence that $C(n,P_{3})$ is order of $n^{c}$ for some constant $c>0$.

We are also interested in some other small graphs which have few vertices and edges. The constant $6$ in Theorem~\ref{thm:666} cannot be directly improved by our results. However, we suspect that the constant $6$ is not the best possible, it will be interesting to improve further. Moreover, motivated by the first result in Theorem~\ref{thm:alon}, one can further consider the graphs with four edges, for instance, the graph consists of a triangle plus a pendant edge and the trees with four edges.

\section*{Acknowledgement}
The authors are extremely grateful to Prof Hong Liu for providing an important idea in the proof of Theorem~\ref{thm:cliques} and offering several kind suggestions which are very helpful to the improvement of the
presentation of this paper.

\bibliographystyle{abbrv}
\bibliography{local}
\end{document}